\documentclass{article}[12pt]
\usepackage[centertags]{amsmath}
\usepackage{amsfonts}
\usepackage{amssymb}
\usepackage{amsthm}
\usepackage{graphicx}
\usepackage{mathtools}
\usepackage{multirow}
\usepackage{varioref}
\usepackage{placeins}
\usepackage{caption}
\usepackage{mathrsfs}
\usepackage{amsmath,amssymb}
\usepackage[T1]{fontenc}
\usepackage{calligra}
\DeclareFontFamily{T1}{calligra}{}
\DeclareFontShape{T1}{calligra}{m}{n}{<->s*[1.44]callig15}{}

\newtheorem{thm}{Theorem}[section]

\newtheorem{example}{Example}
\newtheorem{rmk}{Remark}
\newtheorem{lem}[thm]{Lemma}

\newtheorem{prop}[thm]{Proposition}

\newtheorem*{defn}{Definition}

\usepackage{setspace}

\begin{document}

                                        \title{ \textbf{Modules for Loop Affine-Virasoro Algebras}}

                                        \author{\small{S.Eswara Rao}}
                                        \date{ \small {School of Mathematics, Tata Institute of Fundamental Research, Mumbai, India. \\
                                        senapati@math.tifr.res.in, sena98672@gmail.com }}

                                             \maketitle

\doublespacing
          \begin{abstract}
In this paper we study the representations of loop Affine-Virasoro algebras.  As they have canonical triangular decomposition,we define Verma modules and its irreducible quotients. We give necessary and sufficient condition for an irreducible highest weight module to have finite dimensional weight spaces. We prove that an irreducible integrable module is either an highest weight module or a lowest weight module whenever the canonical cental element acts non-trivially. At the end we construct Affine central operators for each integer and they commute with the action of the Affine Lie algebra.

         \end{abstract}

$~~~~~~~~$\\
$~~~~~~~~$ \textbf{Key words} :Affine-Virasoro algebra, Integrable modules, Affine Central operators.\\
$~~~~~~~~$ \textbf{MSC}: 17B68, 17B67
\doublespacing
\section{ Introduction}
\noindent In recent times the loop algebras are gaining importance and several authors have studied the representations of loop algebras of a simple finite dimensional Lie Algebra. See [1, 15] and references therein. In the theory of Toroidal Lie algebras, the study of irreducible modules with finite dimensional weight spaces has been reduced to the study of modules for the loop affine Lie algebras [4] and [6]. See [16] for the loop Virasoro algebras. On the other hand the semidirect product of Virasoro algebra and affine Lie algebra is very important and occurs in physics literature [12, 13]. They have been also studied from a mathematical point of view [8, 9, 10]. In fact Virasoro Algebra acts on any highest weight module  of an affine Lie algebra except at critical level. See Chapter 12 of [11], and [14]. We do not have this at loop level. So in this paper we consider the loop algebra of the semidirect product of Virasoro and affine Lie algebra.\\
\indent Let $\mathfrak{g}$ be a simple finite dimensional Lie algebra and $\tilde{\mathfrak{g}}$ be the corresponding affine Lie algebra (without the degree derivation).  Let Vir be the Virasoro algebra and let $\tau$ be the semidirect product of $Vir\ltimes \tilde{\mathfrak{g}}$. Now let A be associative commutative finitely generated algebra with unit. We consider the loop algebra $\tau \otimes A : = \tau(A)$ in this paper. The classical loop algebra is when A is a Laurent polynomial in one variale. When A is Laurent polynomial in several variables it is called multi loop algebra and they occur naturally in Toroidal Lie algebras [4].\\
\indent In this paper we consider general A in tune with the recent literature. We now explain the contents of the paper. There is a natural triangular decomposition of $\tau(A)$ induced  by the triangular decomposition of affine Lie algebra and Virasoro algebra. See (1.8). Using this decomposition we define Verma modules and irreducible highest weight modules. The irreducible highest weight modules need not have finite dimensional weight spaces. In Proposition(1.1) we give a necessary and sufficient condition for an irreducible highest weight module to have finite dimensional weight spaces.\\
\indent In Section 2 we define integrable modules and prove that any irreducible integrable module with finite dimensional weight spaces is an highest weight module or a lowest weight module when ever the canonical central element acts non-trivially (Proposition 2.1). In Theorem(2.2) and (2.3) we give a necessary and sufficient condition for a highest weight module to be integrable.\\
\indent Section 3, which is more challenging, we define a category $\mathcal{O}$ (see 3.1 for definition)  of $\tau(A)$ modules which includes Verma modules and its subquotients.  We define Affine central operators which acts on objects of category $\mathcal{O}$  and commutes with the action of $\tilde{\mathfrak{g}}$. These operators are very useful as they will allows us to understand the decomposition of $\tau({A})$ modules with respect to $\tilde{\mathfrak{g}}$.  We define the operators $T_r(a,b)$ for any $r \in \mathbb{Z}, a,b \in A$ and they are motivated by Sugawara operators but they are not Sugawara operators (see chapter 12.8 of [11] for the definition of Sugawara operators).  See [5] where $T_0(a,b)$ is defined and more examples are worked out to show its effectiveness in decomposing a tensor product module.\\
Throughout the paper all vector spaces and tensor products are over complex numbers $\mathbb{C}$. $\mathbb{Z}$ denotes the set of integers.\\
(1.1) Let $\mathfrak{g}$ be any simple finite dimensional Lie Algebra.  Let $\tilde{\mathfrak{g}}=\mathfrak{g} \otimes \mathbb{C}[t,t^{-1}] \oplus \mathbb{C} K $ be the corresponding affine Lie Algebra with Lie bracket $[X \otimes t^n, Y \otimes t^m]=[X,Y]\otimes t^{m+n}+m \delta_{m+n,0}(X,Y)K$ where $X,Y \in \mathfrak{g}, m,n \in \mathbb{Z}$ and ( , ) be a non-degenerate symmetric bilinear form on $\mathfrak{g}$.\\
(1.2) Let $Vir=\displaystyle\bigoplus_{n \in \mathbb{Z}} \mathbb{C}L_n \bigoplus \mathbb{C} C_0$ be the Virasoro Algebra with the bracket $[L_n,L_m]=(m-n)L_{n+m}+\delta_{n+m,0}\frac{n^3-n}{12} C_0$.\\
(1.3)Then consider the Lie Algebra $\bar{\tau}=Vir \rtimes {\tilde{\mathfrak{g}}}$ with  Lie bracket $[L_n,X \otimes t^m]= mX \otimes t^{m+n}$. Let $\tau = \bar{\tau}/ \mathbb{C}(K - C_0)$ where $K$ is center of $\tilde{\mathfrak{g}}$ and $C_0$ is center of Vir.\\
(1.4) We fix a commutative associative finitely generator algebra A with unit.  For any Lie Algebra  $\mathfrak{g'}$, let $\mathfrak{g'}(A)=\mathfrak{g'}\otimes A$ be a Lie algebra with obvious Lie bracket. Denote $X(a)=X \otimes a, X \in \mathfrak{g'}, a \in A$.  $U(\mathfrak{g'})$ always denote the universal enveloping algebra.\\
(1.5)  In this paper we study the Lie Algebra $\tau(A)$ and classify irreducible integrable modules for $\tau(A)$ with finite dimensional weight spaces and center K acting non-trivially.\\
(1.6) For any ideal $I$ of $A$ define the radical ideal $\sqrt{I}=\{a \in A\, | \,a^n \in I, \rm for~ some~ n>0\}$.  It is standard fact that $\sqrt{I}=\displaystyle \bigcap_{i=1}^{k} M_i$, where $M_i$ are distinct maximal ideals of $A$.  By Chineese remainder theorem it follows that $A/\sqrt{I} \cong\bigoplus \mathbb{C}$ (k-copies).\\
(1.7) In particular, for any Lie Algebra $\mathfrak{g'}$ we have a surjective map $\mathfrak{g'} \otimes A \mapsto \bigoplus \mathfrak{g'}$ (k-copies) and the kernel is $\mathfrak{g'} \otimes \sqrt{I}$.\\
(1.8) We now define a triangular decomposition for $\tau(A)$.  Let $\mathfrak{h}$ be Cartan subalgebra of $\mathfrak{g}$  and let $\tilde{\mathfrak{h}}={\mathfrak{h}}\bigoplus\mathbb{C}K$ and $\bar{\mathfrak{h}} = \tilde{\mathfrak{h}} \oplus \mathbb{C} L_0$.  Let $N^{-}\bigoplus\tilde{\mathfrak{h}}\bigoplus N^{+}$ be the standard decomposition of $\tilde{\mathfrak{g}}$.\\
Let $L^{-}=\displaystyle\bigoplus_{n < 0} \mathbb{C}L_n, L^{+}=\displaystyle\bigoplus_{n >0} \mathbb{C}L_n$, $L^{0}=\mathbb{C}L_0 \bigoplus \mathbb{C}K$
 and $Vir=L^{-} \bigoplus L^{0} \bigoplus L^{+}$.  Now consider the triangular decomposition of $\tau(A)$ as $\tau^{-}(A)\bigoplus \tau^{0}(A)\bigoplus \tau^{+}(A)$ where $\tau^{-}(A)=N^{-}(A)\bigoplus L^{-}(A)$\\
$\tau^{+}(A)=N^{+}(A)\bigoplus L^{+}(A)$\\
$\bar{\mathfrak{h}}(A)=\tilde{\mathfrak{h}}(A)\bigoplus \mathbb{C}L_{0}(A) = \tau^{0}(A)$\\
(1.9) Let $\psi:\tau^{0}(A)\rightarrow \mathbb{C}$ be a linear map and consider the Verma module $M(\psi)=\tau(A){\substack{\bigotimes \\\tau^{0}(A)\oplus \tau^{+}(A)} }\mathbb{C}v$ where $\mathbb{C}v$ is a one dimensional representation of $\tau^{+}(A)\oplus \tau^{0}(A)$  and $\tau^{+}(A)$  acts trivially and $\tau^{0}(A)$ acts via $\psi$.  By standard arguments it follows that $M(\psi)$ has a unique irreducible quotient say $V(\psi)$.  Clearly $M(\psi)$ is weight module with respect to $\bar{\mathfrak{h}}=\tilde{\mathfrak{h}}\bigoplus \mathbb{C}L_0$.  $M(\psi)$ does not have finite dimensional weight spaces whenever $A$ is infinite dimensional.  $V(\psi)$ may have finite dimensional weight spaces depending on $\psi$.  We will now investigate when $V(\psi)$ has finite dimensional weight spaces.\\
\begin{prop} \rm
$V(\psi)$ has finite dimensional weight spaces with respect to $\bar{\mathfrak{h}}$ if and only if there exists a co-finite Ideal $I$ of $A$ such that $\psi(\bar{\mathfrak{h}} \otimes I)=0$.  In that case $\tau(I).V(\psi)=0$ 
\end{prop}

\begin{proof}
Suppose $V(\psi)$ has finite dimensional weight spaces.  As in the proof of Lemma 2.3 of [6] there exists a co-finite ideal of $I_1$ of $A$ such that $\tilde{\mathfrak{h}}(I_1)v=0$.\\
Now consider $I_2=\{a \in A \,\,  |\,\, L_{-1}(a)v=0$\} and as in [6] it is easy to prove $I_2$ is a co-finite idel of $A$.\\
Consider $L_1L_{-1}(I_2)v=L_{-1}(I_2)L_1v - 2L_0(I_2)v$. Thus we have $L_0(I_2) v = 0$. Let $I=I_1I_2$ which is also co-finite ideal of $A$ by Lemma 2.2 of [6].  We have $\bar{\mathfrak{h}}(I)v=0$.
Now suppose there exists a co-finite ideal $I$ of $A$ such that $\bar{\mathfrak{h}}(I)v=0$.  We use the technique of [6].  We prove that $X_{-\alpha}(I)v=0, \alpha \geq 0$ by induction on the height of $\alpha$ (where $X_{\alpha}$ is a root vector of root $\alpha$ inside $\tilde{\mathfrak{g}}$). The assertion is clear for $\alpha  = 0$ by assumption that $\bar{\mathfrak{h}}(I)v = 0$.  It is proved in the Proposition 2.4 of [6] that $X_{-\alpha}(I)v$ is highest weight vector for $\tilde{\mathfrak{g}}(A)$.  Now consider for $n>0, L_n(a)X_{-\alpha}(I)v=X_{-\alpha}(I)L_n(a)v+[L_n(a),X_{-\alpha}(I)]v$.  The first term zero and the second term is zero by induction as the height decreases.  This proves $X_{-\alpha}(I)v$ is a highest weight vector for $\tau(A)$ inside an irreducible module $V(\psi)$.  Thus proving $X_{-\alpha}(I)v=0$. Similarly we can prove that $\text{Vir}(I)v = 0$. Now consider $W=\{w \in V(\psi) \,\, | \,\, \tau(I)v=0\}$ which can be verified to be $\tau(A)$-module.  From above we know that $v \in W$.  Hence $W$ is a non-zero submodule of $V(\psi)$ which means $W=V(\psi)$.  Thus it follows that $\tau(I)V(\psi)=0$.  Now it is easy to see that $V(\psi)$ has finite dimensional weight spaces as it is a module for $\tau(A/I)$.  See [6] for details.
\end{proof}
\section{Integrable modules for non-zero level}
In this section we define integrable modules for $\tau(A)$ and classify them when $K$ acts non-trivially.  
\begin{defn} 
A module $V$ of $\tau(A)$ is called integrable if the following holds.\\
$~~~~$ (i) $V=\bigoplus_{\lambda \in \bar{\mathfrak{h}}^*}V_{\lambda}$ where $V_{\lambda}=\{v \in V\,\,|\,\,hv=\lambda(h)v, \forall h \in \bar{\mathfrak{h}},\text{dim} \, V_{\lambda}< \infty \}$\\\,
$~~~~$  (ii) For any $v \in V, \alpha \in \Delta^{real}$, and $a \in A$, there is exists $N=N(\alpha,v,a) $ such that $X_{\alpha}^N(a).v=0$ where $X_{\alpha}$ is a root vector of $\mathfrak{g}_{\alpha}$ .
 \end{defn}
\begin{prop} \rm
Suppose $V$ is an irreducible integrable module for $\tau(A)$.  Suppose the  central element $K$ acts non-trivially.  Then $V$ is an highest weight module or a lowest weight module.
\end{prop}
\begin{proof}
It is a standard fact that $K$ acts by an integer in an integrable module for an affine Lie algebra.  We can assume $K$ acts by a positive integer. If $K$ acts as a negative integer we get a lowest weight module.

Let $V=\bigoplus_{\lambda \in \bar{\mathfrak{h}}^*}V_{\lambda}, \,\text{dim} \, V_{\lambda}< \infty$.  Now by arguments similar to the proof of Proposition 2.4 of [2], it follows that $V$ is a highest weight module.  The proof in [2]  is given only for Laurent polynomial algebra in several variables.  But the proof works for any algebra $A$.  Just note that $L_n (a), a \in A$, corresponds to root $n\delta$ where $\delta$ is standard null root of the affine Lie algebra. See also [3] for similar result.
\end{proof}
  We will now classify integrable highest weight modules $V(\psi)$.  In other words we will indicate for what $\psi$  the module $V(\psi)$ is integrable.  We already know that there is a co-finite Ideal $I$ of $A$ such that $(\bar{\mathfrak{h}}\otimes I)V(\psi)=0$ since we are assuming the weight spaces are finite dimensional(see Proposition 1.1)

\begin{thm} \rm
Suppose $V(\psi)$ is integrable for $\tau(A)$.  Then there is co-finite ideal $I$ of $A$ such that $\psi(L_0(I))=0$ and $\psi(\tilde{\mathfrak{h}} \otimes \sqrt{I})=0$
\end{thm}
\begin{proof}
We already know that there is co-finite ideal $I$ of $A$ such that \\ $\psi(\tilde{\mathfrak{h}}\otimes I \bigoplus L_0 \otimes I)=0$.  Now let $M$ be the module generated by the highest weight vector for the Lie algebra $\tilde{\mathfrak{g}}(A) \oplus \mathbb{C} L_0$.  Then $M$ is an integrable highest weight module and let $\bar{M}$ be the unique irreducible quotient for Lie algebra  $\tilde{\mathfrak{g}}(A)\bigoplus \mathbb{C}L_0$.  Then by Theorem 3.4 of [6] it follows $\tilde{\mathfrak{g}}(\sqrt{I}).\bar{M}=0$.  In particular $(\tilde{\mathfrak{h}}\otimes \sqrt{I})\bar{M}=0$.  This completes the proof. 
\end{proof}
We will now prove the converse.\\
(2.2) Suppose $J=M_1\cap M_2 \cap \ldots \cap M_k$ where each $M_i$ is a maximal ideal and they are distinct.  By Chinese remainder theorem $A/J \cong \bigoplus_{i=1}^k \frac{A}{M_i} \cong \bigoplus \mathbb{C}(k-copies)$.  In particular for any Lie Algebra $\mathfrak{g'}$ we have $\mathfrak{g'}\otimes A\rightarrow \bigoplus \mathfrak{g'}$ (k-copies) is a surjective map with kernal $\mathfrak{g'}\otimes J$.  For a co-finite ideal $I$ we have the radical ideal $\sqrt{I}=J$.  It is a standard fact that $J$ is intersection of maximal ideals.  So that $ \tilde{\mathfrak{h}} \otimes A/J \cong \bigoplus \tilde{\mathfrak{h}} $(k-copies).\\
(2.3) We say $\psi: \bar{\mathfrak{h}} \otimes A \rightarrow \mathbb{C}$ is dominant integral if there is a co-finite ideal $I$ such that $\psi (\tilde{\mathfrak{h}}\otimes \sqrt{I})=0$, $\psi_{i}( \bar{\mathfrak{h}})$ is dominant integral for $1 \leq i \leq k$, where
$\psi_i$ denote the restricion of $\psi$ on the $i$-th piece of $\bigoplus{\bar{\mathfrak{h}}}$.

\begin{thm} \rm
Suppose $V(\psi)$ is an irreducible highest weight module for $\tau(A)$ with finite dimensional weight spaces.  Assume that there is a co-finite ideal $I$ of $A$ such that $\psi(L_0\otimes I)=0$ and $\psi(\tilde{\mathfrak{h}}\otimes \sqrt{I})=0$ and $\psi$ is dominant integral with respect to $\sqrt{I}$. Then $V(\psi)$ is integrable.
\end{thm}
Before proving the theorem we need few Lemmas.\\
\begin{lem} \rm
With above notation let $M=U(\tilde{\mathfrak{g}}\otimes A)v$, where $v$ is a highest weight vector of $V(\psi)$.  Then $M$ is irreducible module for $(\tilde{\mathfrak{g}} \otimes A \bigoplus \mathbb{C}L_0)$.
\end{lem}

\begin{proof}
Suppose $M$ is not irreducible for $\tilde{\mathfrak{g}}(A)$, then there exists a highest weight vector $w \not\in \mathbb{C}v$.  We can assume that $w$ is of maximal weight.\\
\underline{claim}: $w$ is a highest weight for $\tau(A)$\\
we need to prove that $L_n(a).w=0$ for $n>0$ and $a \in A$.   It is sufficient to check this for $L_1(a)$ and $L_2(b)$, $a,b \in A$ as they generate $L^+(A)$.  First notice that $L_1(a)w$ and $L_2(a)w$  belong to $M$.  It is trivial checking that $L_1(a)w$ and $L_2(a)w$  are $\tilde{\mathfrak{g}}(A)$ highest weight vector.  Since $w$ is maximal weight $L_1(a)w=0$ unless $w \in M_{\psi-\delta}^+$ and $L_2(a)w=0$ unless $w \in M_{\psi-2\delta}^+$, where + denotes highest weight vectors for $\tilde{\mathfrak{g}}(A)$.\\
\underline{subclaim}: $ M_{\psi-n\delta}^+=0, n>0.$\\
Assuming the subclaim we complete the proof of the Lemma.  So we have $w$ is a highest weight vector for $\tau(A)$ which is contradiction  as $V(\psi)$ being irreducible.  This proves the Lemma.\\
Proof of the subclaim:
Since $\psi( \bar{\mathfrak{h}}(\sqrt{I}))$ is zero, using arguments similar to Proposition 1.1 one can prove that $\tilde{\mathfrak{g}}(\sqrt{I}).V(\psi)=0$ but $\sqrt{I}=J$ is a radical ideal and as noted earlier $J=\bigcap_{i=1}^{k}M_i$ where $M_i$ are distinct maximal ideals.  As noted in 2.2 $V(\psi)$ is a module for  $\bigoplus \tilde{ \mathfrak{g}}$ (k-copies).  In particular $M$ is an highest weight module for $\bigoplus \tilde{\mathfrak{g}}$.  It is standard fact that there exists highest weight module $V(\lambda_i)$ for $\tilde{\mathfrak{g}}$ such that $M \cong \displaystyle{\bigotimes_{i=1}^{k}V(\lambda_i)}$.  Let $\tilde{\mathfrak{g}}$ be the $i^{th}$ copy of $\bigoplus\tilde{\mathfrak{g}}$.  Then one can take $V(\lambda_i)$ as $\tilde{\mathfrak{g}}$ module generated by $v$.  Further $V(\lambda_i)$ is also module for $\tilde{\mathfrak{g}}\bigoplus \mathbb{C}L_0$.  Let $\Omega_i(1\leq i\leq k)$ be the Casimir operator for $\tilde{\mathfrak{g}}$.  Then $\sum \Omega_i$ is the Casimir Operator for $\bigoplus \tilde{\mathfrak{g}}$.  It is well known that [see [11], Lemma 9.8 a] $\Omega_i$ acts as $|\lambda_i+\rho|^2-|\rho|^2$ on $V(\lambda_i)$.  Thus $\sum \Omega_i$ acts as $\sum_{i}|\lambda_i+\rho|^2-|\rho|^2$.  Suppose $M_{\psi-n\delta}^{+}\neq 0$. Then  $\sum \Omega_i$ acts on the above as $\sum_{i}|\lambda_i+\rho - n \delta|^2-|\rho|^2$.  It follows that $\sum n (\lambda_i+\rho, \delta)=0$ but by assumption each $\lambda_i$ is dominant integral and hence the above equation is not possible.  This proves the subclaim.\\
Proof of Theorem 2.3: \\
We know that $\otimes V(\lambda_i)\cong M$ is irreducible  as $\mathfrak{g}(A)$ and each $\lambda_i$ is dominant integral.  Hence $\otimes V(\lambda_i)$ is a integrable for $ \mathfrak{g}(A)$ (see corollary 10.4 of [11]).  In particular each $X_{\alpha}(a), \alpha \in {\Delta}^{real}$ is locally nilpotent on the generator $v$.  Further each $X_{\alpha}(a)$ is locally nilpotent on $\tau(A)$.  Hence by [lemma(3.4(b)) of 11] we see that each $X_{\alpha}(a)$ is locally nilpotent on $V(\psi)$.  This completes the proof of the theorem 2.3.
\end{proof}

\section{Affine central operators}
 In this section we study $\tau(A)$ modules in category $\mathcal{O}$.  We construct affine central operators acting on objects of $\mathcal{O}$ and commute with affine Lie algebra $\tilde{\mathfrak{g}}$.  We need some notation for that: \\
(3.1) A module $V$ of $\tau(A)$ is said to be in category $\mathcal{O}$ if the following holds:
\begin{enumerate}
\item $V$ is weight module for $\tau(A)$ with respect to Cartan subalgebra $\bar{\mathfrak{h}}$ and has finite dimensional weight spaces.
\item For any $v \in V$ and $a \in A$, we have $(X_{\alpha} \otimes a) v = 0$ for ht $\alpha >> 0,  \alpha \in \Delta^{+}$ and $X_{\alpha} \in \tau(A)_{\alpha}$.
\end{enumerate}
(3.2) We recall some known facts from [11].
Let $\alpha_0= - \beta+\delta$, where $\beta$ is a highest root of $\mathfrak{g}$ and $\delta$ is the standard null root of $\tilde{\mathfrak{g}}$.
 Let ( , ) be the non-degenerate bilinear form on $\bar{\mathfrak{h}}$.
  Let $\gamma: \bar{\mathfrak{h}}\rightarrow {(\bar{\mathfrak{h}})}^*$ be such that $\gamma(h_1)(h_2)=(h_1, h_2), h_1,h_2 \in \bar{\mathfrak{h}}$.  This isomorphism induces a non degenerate bilinear form on $ {(\bar{\mathfrak{h}})}^*$.\\
 Let $\rho \in (\bar{\mathfrak{h}})^*$ such that $(\rho, \alpha_i)=\frac{1}{2}(\alpha_i, \alpha_i) \,\,\forall \alpha_i$.  Let $\bar{\rho}= \rho|_{\mathfrak{h}}$. Recall $\delta\in (\bar{\mathfrak{h}})^* \ni \delta(\mathfrak{h})=0, \delta(K)=1~\text{and}~\delta(d)=0$.  Then $\rho=\bar{\rho}+h^{\vee} \Lambda_0$ (see 6.2.8 of [11]) where $h^{\vee}$ is the dual Coxeter number of $\mathfrak{g}$.  Note that $\gamma(d) = \Lambda_0$ (see 6.23. of [11]).\\
 Let $\Delta=\{\alpha+n\delta, m\delta, 0\neq m \in \mathbb{Z}, n\in \mathbb{Z}, \alpha \in \overset{\circ} {\Delta}\}$. Let $\Delta^{+}$ be the positive roots.  Let $\{h_i: 1 \leq i \leq \text{dim}\,\, \mathfrak{h}\}$ be a basis of $\mathfrak{h}$. Let $\{h^{i}: 1 \leq i \leq \text{dim} \,\, \mathfrak{h}\}$ be a dual basis of $\mathfrak{h}$ with respect to the basis $\{h_i\}$.  Then $\{h_i, d, K\},\{h^i, K, d\}$ is a dual basis of $\bar{\mathfrak{h}}$.  For each root $\alpha \in \overset{\circ} {\Delta}$, let $x_{\alpha}\in \mathfrak{g}_{\alpha}, x_{-\alpha}\in \mathfrak{g}_{-\alpha}$ such that $(x_{\alpha},x_{-\alpha})=1$, then $[x_{\alpha},x_{-\alpha}]=\gamma(\alpha)$ see theorem 2.2(e) of [11].  Then the Casimir operator for $\mathfrak{g}$ is defined by $\Omega=2 \gamma^{-1}(\rho)+\sum_ih_{i}h^i+2Kd+2\sum_{ \alpha \in\overset{\circ} {\Delta}}\sum_{n>0}x_{-\alpha}\otimes t^{-n}x_{\alpha}\otimes t^n+2\sum_{n>0}\sum_{i}h_{i}\otimes t^{-n}h^{i}\otimes t^n+2\sum_{ \overset{\circ} {\Delta}_+}x_{-\alpha}x_{\alpha}$.  See 12.8.3 of [11] and note that $\gamma^{-1}(\rho)=\gamma^{-1}(\bar{\rho})+h^{\vee}d$ (later $d$ be identified with $L_0$).\\
 (3.3) Define for $a,b \in A$\\
  $\Omega_{a,b}^1=\sum_{\alpha \in \overset{\circ} {\Delta}}\sum_{n>0}x_{-\alpha}\otimes t^{-n}(a)x_{\alpha}\otimes t^{n}(b)$\\
  $\Omega_{a,b}^2=\sum_{i}\sum_{n>0}h_{i}\otimes t^{-n}(a)h^{i} \otimes t^{n}(b)$ \\
  $\Omega_{a,b}^3=\displaystyle{\sum_{\alpha \in \overset{\circ} {\Delta}^+}x_{-\alpha}(a)x_{\alpha}(b)}$\\
   then define $\Omega(a,b)=2\gamma^{-1}(\rho)(ab)+\sum h_i(a)h^{i}(b)+K(a)d(b)+K(b)d(a)+\Omega_{a,b}^1+\Omega_{b,a}^1+\Omega_{a,b}^2+\Omega_{b,a}^2+\Omega_{a,b}^3+\Omega_{b,a}^3$.  This is exactly the operator defined in 2.4 of [5].  There it is defined for any symmetrizable Kac-Moody Lie algebra.  But here we defined only for the affine Kac-Moody Lie algebra. These operators should be seen to be vectors of competion of $\tau(A)$.\\
 \begin{prop} \rm  $[\Omega(a,b), {\tilde{\mathfrak{g}}}]=0$  on objects of $\mathcal{O}$.
\end{prop}
\begin{proof}
 The proof is exactly as given in Theorem 2.5 of [5] applied to the affine case.
 \end{proof}

   \begin{defn} 
   An operator $z$ acting on objects of $\mathcal{O}$ is called affine central operator if $z$ commutes with $\tilde{\mathfrak{g}}$.
   \end{defn}
For example $\Omega(a,b)$ is an affine central operator.\\
(3.4) Let for $j\neq 0, T_j (a,b)=\frac{-1}{j}[L_j,\Omega(a,b)]$ and $T_0(a,b)=\Omega(a,b)$.
These operators are motivated by Sugawara operators. But they are not Sugawara operators. Sugawara operators are part of these operators when $a = b =1$.
\begin{thm} \rm
$T_j (a,b)$ for $j \in \mathbb{Z}$ is an affine central operator
\end{thm}
\begin{proof}
We can assume $j \neq 0$ as we already know for $j=0$.  Let $x \in \mathfrak{g},k \in\mathbb{Z}, 0 \neq j \in \mathbb{Z}$.\\
Consider $[x(k),T_j (a,b)]=\frac{-1}{j}[x(k), [L_j ,\Omega(a,b)]]$\\
$=\frac{1}{j}[L_{j}, [\Omega(a,b), x(k)]] +\frac{1}{j}[\Omega(a,b), [x(k),  L_{j}]]$.\\
The first term is zero by Proposition 3.1.  The second term is also zero by the same proposition by noting that $[x(k),L_j]=-kx(j+k)$.  This completes the F.
\end{proof}
We now give an expression for $T_j (a,b)$.  We need some lemmas for that. \\
\begin{lem} \rm \label{l1}
(a) $n \in \mathbb{Z},0 \neq j \in \mathbb{Z}, \sum_{\alpha \in \overset{\circ} {\Delta}}x_{-\alpha}\otimes t^{-n}(a) x_{\alpha}\otimes t^{n+j}(b)= \sum_{\alpha \in \overset{\circ} {\Delta}}x_{\alpha}\otimes t^{n+j}(b) x_{-\alpha}\otimes t^{-n}(a) $\\
(b) $\sum_{i}h_{i}\otimes t^{-n}(a) h^{i}\otimes t^{n+j}(b)=\sum_{i}h^{i}\otimes t^{n+j}(b) h_{i}\otimes t^{-n}(a)$
\end{lem}
\begin{proof}
(\textbf{a})  we have $LHS=RHS+\sum_{\alpha \in \overset{\circ} {\Delta}}\gamma^{-1}(\alpha) \otimes t^j(ab)$, but $\sum_{\alpha \in \overset{\circ} {\Delta}}\gamma^{-1}(\alpha) \otimes t^j(ab)$ is zero as $\gamma^{-1}(-\alpha) = -\gamma^{-1}{(\alpha)}$ and hence  (a) is proved and (b) is similar
\end{proof}
\begin{lem} \rm  \label{l2}
(a) $0 \neq j \in \mathbb{Z}, \\ \sum_{n \in \mathbb{Z}}\sum_{\alpha \in \overset{\circ} {\Delta}}x_{-\alpha}\otimes t^{-n}(a) x_{\alpha}\otimes t^{n+j}(b)= \sum_{n \in \mathbb{Z}}\sum_{\alpha \in \overset{\circ} {\Delta}}x_{-\alpha}\otimes t^{-n}(b) x_{\alpha}\otimes t^{n+j}(a) $\\
(b) $0 \neq j \in \mathbb{Z}, \sum_{i}\sum_{n \in \mathbb{Z}}h_{i}\otimes t^{-n}(a) h^{i}\otimes t^{n+j}(b)=\sum_{i}\sum_{n \in \mathbb{Z}}h_{i}\otimes t^{-n}(b) h^{i}\otimes t^{n+j}(a)$
\end{lem}
\begin{proof}
proof of (a) Follows from Lemma \ref{l1}(a) by sending $n+j$ to $-n$ and noting that when $\alpha$ varies in $\overset{\circ} {\Delta}$ then so is $-\alpha$.\\
(b) The proof is similar by noting that the sum is independent of the dual basis.
\end{proof}
\begin{lem} \rm \label{l3}
For $0 \neq j \in \mathbb{Z}$, \\ let $ B = \sum_{\alpha \in \overset{\circ} {\Delta}}\sum_{n \in \mathbb{Z}}nx_{-\alpha}\otimes t^{-n}(a) x_{\alpha}\otimes t^{n+j}(b)+ \sum_{\alpha \in \overset{\circ} {\Delta}}\sum_{n \in \mathbb{Z}}nx_{-\alpha}\otimes t^{-n}(b) x_{\alpha}\otimes t^{n+j}(a)$\\
  then $ B=-j\sum_{\alpha \in \overset{\circ} {\Delta}}\sum_{n \in \mathbb{Z}}x_{-\alpha}\otimes t^{-n}(a) x_{\alpha}\otimes t^{n+j}(b) $
\end{lem}
\begin{proof}
By sending $-n ~to~ n+j$ in both sums of $B$ we see that\\
$B=\sum_{\alpha \in \overset{\circ} {\Delta}}\sum_{n \in \mathbb{Z}}(-n-j)x_{-\alpha}\otimes t^{n+j}(a) x_{\alpha}\otimes t^{-n}(b)+ \\ \sum_{\alpha \in \overset{\circ} {\Delta}}\sum_{n \in \mathbb{Z}}(-n-j)x_{-\alpha}\otimes t^{n+j}(b) x_{\alpha}\otimes t^{-n}(a)$\\
Then by Lemma \ref{l1}(a) we see that \\
$B=\sum_{\alpha \in \overset{\circ} {\Delta}}\sum_{n \in \mathbb{Z}}(-n-j)x_{-\alpha}\otimes t^{-n}(b) x_{\alpha}\otimes t^{n+j}(a)+\sum_{\alpha \in \overset{\circ} {\Delta}}\sum_{n \in \mathbb{Z}}(-n - j)x_{-\alpha}\otimes t^{-n}(a) x_{\alpha}\otimes t^{n+j}(b)$\\
$=-B-j\sum_{\alpha \in \overset{\circ} {\Delta}}\sum_{n \in \mathbb{Z}}x_{-\alpha}\otimes t^{-n}(b) x_{\alpha}\otimes t^{n+j}(a)-j\sum_{\alpha \in \overset{\circ} {\Delta}}\sum_{n \in \mathbb{Z}}x_{-\alpha}\otimes t^{-n}(a) x_{\alpha}\otimes t^{n+j}(b)$\\
$~~~~~$ $= -B-2j\sum_{\alpha \in \overset{\circ} {\Delta}}\sum_{n \in \mathbb{Z}}x_{-\alpha}\otimes t^{-n}(a) x_{\alpha}\otimes t^{n+j}(b)$ by Lemma \ref{l2}(a).\\
Thus $2B=-2j\sum_{\alpha \in \overset{\circ} {\Delta}}\sum_{n \in \mathbb{Z}}x_{-\alpha}\otimes t^{-n}(a) x_{\alpha}\otimes t^{n+j}(b)$\\
This completes the proof of the lemma.
\end{proof}
\begin{lem} \rm \label{l4}
(a) For $0 \neq j \in \mathbb{Z}$, $[L_j, \Omega_{a,b}^1+\Omega_{b,a}^1]=-j\sum_{\alpha \in \overset{\circ} {\Delta}}\sum_{n \in \mathbb{Z}}x_{-\alpha}\otimes t^{-n}(a) x_{\alpha}\otimes t^{n+j}(b)$;\\
(b)\,$0 \neq j \in \mathbb{Z}$, $[L_j, \Omega_{a,b}^2+\Omega_{b,a}^2]=-j\sum_{i}\sum_{n \in \mathbb{Z}}h_{i}\otimes t^{-n}(a) h^{i}\otimes t^{n+j}(b).$
\end{lem}
\begin{proof}
Consider $[L_j, \Omega_{a,b}^1]=\sum_{\alpha \in \overset{\circ} {\Delta}}\sum_{n >0}-nx_{-\alpha}\otimes t^{-n+j}(a) x_{\alpha}\otimes t^{n}(b)+\sum_{\alpha \in \overset{\circ} {\Delta}}\sum_{n >0}nx_{-\alpha}\otimes t^{-n}(a) x_{\alpha}\otimes t^{n+j}(b).$\\
Now by replacing n by $-n$ in the first sum using Lemma \ref{l1}(a) we see that \\
$[L_j, \Omega_{a,b}^1]= \sum_{\alpha \in \overset{\circ} {\Delta}}\sum_{n<0}nx_{-\alpha}\otimes t^{-n}(b) x_{\alpha}\otimes t^{n+j}(a)+\sum_{\alpha \in \overset{\circ} {\Delta}}\sum_{n \geq 0}nx_{-\alpha}\otimes t^{-n}(a) x_{\alpha}\otimes t^{n+j}(b).$\\
Now we have
\begin{align*}[L_j, \Omega_{a,b}^1+\Omega_{b,a}^1]&=\sum_{\alpha \in \overset{\circ} {\Delta}}\sum_{n \in \mathbb{Z}}nx_{-\alpha}\otimes t^{-n}(a) x_{\alpha}\otimes t^{n+j}(b)+\sum_{\alpha \in \overset{\circ} {\Delta}}\sum_{n \in \mathbb{Z}}nx_{-\alpha}\otimes t^{-n}(b) x_{\alpha}\otimes t^{n+j}(a)\\
&= B (\rm {as~ defined~ in ~Lemma ~\ref{l3}})\\
&=-j\sum_{\alpha \in \overset{\circ} {\Delta}}\sum_{n \in \mathbb{z}}x_{-\alpha}\otimes t^{-n}(a) x_{\alpha}\otimes t^{n+j}(b) (\rm {by ~Lemma~ \ref{l3}})
\end{align*}
(b) proof is similar to (a)
\end{proof}
\begin{prop} \rm
$0 \neq j \in \mathbb{Z},\\ T_j (a,b)=\frac{-1}{j}[L_j,\Omega (a,b)]$ \\
$~~~~~~~~~~$ $=\sum_{\alpha \in \overset{\circ} {\Delta}}\sum_{n \in \mathbb{Z}}x_{-\alpha}\otimes t^{-n}(a) x_{\alpha}\otimes t^{n+j}(b)+$\\
 $~~~~~~~~~~~~~$ $\sum_{i}\sum_{n \in \mathbb{Z}}h_{i}\otimes t^{-n}(a) h^{i}\otimes t^{n+j}(b)+K(a)L_j(b)+K(b)L_j(a)+2h^{\vee}L_j(ab)$
\end{prop}
\begin{proof}
Follows from Lemma \ref{l4} and by noting the following brackets:
$[L_j, \gamma^{-1}(\rho) (ab)] = -h^{\vee} j L_j (ab)$ and $[L_j, \Omega^{3}_{a, b} + \Omega^{3}_{b, a}] = 0$.
\end{proof}
We will now record the following useful bracket

\begin{prop} \rm 
$0 \neq j \in \mathbb{Z}, k \in \mathbb{Z}, [L_k,T_j (a,b)]=(j-k)T_{j+k}(a,b)-\delta_{j+k,0}\frac{k^3-k}{6}(dim \mathfrak{g})K(ab) + \delta_{j +k, 0} \frac{k^3-k}{12}(2 K(a)K(b) + 2 h^{\vee} K(ab)).$
\end{prop}
\begin{proof}
For $j+k \neq 0$, the Proposition follows easily by earlier arguments.\\
Let $j+k=0$ and assume that $k > 0$ write $\sum_{\alpha \in \overset{\circ} {\Delta}}\sum_{n \in \mathbb{Z}}x_{-\alpha}\otimes t^{-n}(a) x_{\alpha}\otimes t^{n+j}(b)$\\
as $\sum_{\alpha \in \overset{\circ} {\Delta}}\sum_{n \geq 0}x_{-\alpha}\otimes t^{-n}(a) x_{\alpha}\otimes t^{n+j}(b)+ \sum_{\alpha \in \overset{\circ} {\Delta}}\sum_{n < 0}x_{\alpha}\otimes t^{n+j}(b) x_{-\alpha}\otimes t^{-n}(a)$\\
change n by $-n$ in the second sum and apply $L_k$.\\
Thus
$[L_k, \sum_{\alpha \in \overset{\circ} {\Delta}}\sum_{n \in \mathbb{Z}}x_{-\alpha}\otimes t^{-n}(a) x_{\alpha}\otimes t^{n+j}(b)]$\\
$=[L_k, \sum_{\alpha \in \overset{\circ} {\Delta}}\sum_{n \geq 0}x_{-\alpha}\otimes t^{-n}(a) x_{\alpha}\otimes t^{n+j}(b)+[L_k, \sum_{\alpha \in \overset{\circ} {\Delta}}\sum_{n>0}x_{\alpha}\otimes t^{-n+j}(b) x_{-\alpha}\otimes t^{n}(a)]$\\
$=\sum_{\alpha \in \overset{\circ} {\Delta}}\sum_{n \geq 0}(-n)x_{-\alpha}\otimes t^{-n+k}(a) x_{\alpha}\otimes t^{n+j}(b)+$\\
$~~~~$ $\sum_{\alpha \in \overset{\circ} {\Delta}}\sum_{n \geq 0}(n+j)x_{-\alpha}\otimes t^{-n}(a) x_{\alpha}\otimes t^{n+j+k}(b)+$\\
$~~~~$ $\sum_{\alpha \in \overset{\circ} {\Delta}}\sum_{n > 0}(-n+j)x_{\alpha}\otimes t^{-n+j+k}(b) x_{-\alpha}\otimes t^{n}(a)+$\\
$~~~~$ $\sum_{\alpha \in \overset{\circ} {\Delta}}\sum_{n> 0} nx_{\alpha}\otimes t^{-n+j}(b) x_{-\alpha}\otimes t^{n+k}(a)$\\
In the first sum replace $-n+k$ by $-n$, in the fourth sum replace $-n+j$ by $-n$ then we have\\
$=\sum_{\alpha \in \overset{\circ} {\Delta}}\sum_{n+k \geq 0}-(n+k)x_{-\alpha}\otimes t^{-n}(a) x_{\alpha}\otimes t^{n+j+k}(b)+$\\
$~~~~$ $\sum_{\alpha \in \overset{\circ} {\Delta}}\sum_{n \geq 0}(n+j)x_{-\alpha}\otimes t^{-n}(a) x_{\alpha}\otimes t^{n+j+k}(b)+$\\
$~~~~$ $\sum_{\alpha \in \overset{\circ} {\Delta}}\sum_{n> 0}(-n+j)x_{-\alpha}\otimes t^{-n+j+k}(b) x_{\alpha}\otimes t^{n}(a)+$\\
$~~~~$ $\sum_{\alpha \in \overset{\circ} {\Delta}}\sum_{n+j>0}(n+j)x_{-\alpha}\otimes t^{-n}(b) x_{\alpha}\otimes t^{n+j+k}(a),$\\
use the fact that $j+k=0$\\
 $=\sum_{\alpha \in \overset{\circ} {\Delta}}\sum_{n \geq 0}-(n+k)x_{-\alpha}\otimes t^{-n}(a) x_{\alpha}\otimes t^{n}(b)+$\\
$~~~~$ $\sum_{\alpha \in \overset{\circ} {\Delta}}\sum_{-k \leq n < 0}-(n+k)x_{-\alpha}\otimes t^{-n}(a) x_{\alpha}\otimes t^{n}(b)+$\\
$~~~~$ $\sum_{\alpha \in \overset{\circ} {\Delta}}\sum_{n \geq 0}(n+j)x_{-\alpha}\otimes t^{-n}(a) x_{\alpha}\otimes t^{n}(b)+$\\
$~~~~$ $\sum_{\alpha \in \overset{\circ} {\Delta}}\sum_{n>0}(-n+j)x_{\alpha}\otimes t^{-n}(b) x_{-\alpha}\otimes t^{n}(a)+$\\
$~~~~$ $\sum_{\alpha \in \overset{\circ} {\Delta}}\sum_{n>0}(n+j)x_{\alpha}\otimes t^{-n}(b) x_{-\alpha}\otimes t^{n}(a)-$\\
$~~~~$ $\sum_{\alpha \in \overset{\circ} {\Delta}}\sum_{0 < n\leq -j}(n+j)x_{\alpha}\otimes t^{-n}(b) x_{-\alpha}\otimes t^{n}(a)$\\
$=2j\sum_{\alpha \in \overset{\circ} {\Delta}}\sum_{n \geq 0}x_{-\alpha}\otimes t^{-n}(a) x_{\alpha}\otimes t^{n}(b)+$\\
$~~~~$ $2j\sum_{\alpha \in \overset{\circ} {\Delta}}\sum_{n > 0}x_{\alpha}\otimes t^{-n}(b) x_{-\alpha}\otimes t^{n}(a)+$\\
$~~~~$ $\sum_{\alpha \in \overset{\circ} {\Delta}}\sum_{-k \leq n<0}-(n+k)x_{-\alpha}\otimes t^{-n}(a) x_{\alpha}\otimes t^{n}(b)+$\\
$~~~~$ $\sum_{\alpha \in \overset{\circ} {\Delta}}\sum_{0 \lneqq n \leq -j}-(n+j)x_{\alpha}\otimes t^{-n}(b) x_{-\alpha}\otimes t^{n}(a )$\\
$~~~~$ $= 2j(\Omega_{a,b}^1+\Omega_{b,a}^1)+2j \sum_{\alpha \in \overset{\circ} {\Delta}}x_{-\alpha}(a)x_{\alpha}(b)+$\\
$~~~~$ $\sum_{\alpha \in \overset{\circ} {\Delta}}\sum_{0 < n \leq -j}(n+j)x_{-\alpha}\otimes t^{n}(a) x_{\alpha}\otimes t^{-n}(b)+$\\
$~~~~$ $\sum_{\alpha \in \overset{\circ} {\Delta}}\sum_{0 <n\leq -j}-(n+j)x_{\alpha}\otimes t^{-n}(b) x_{-\alpha}\otimes t^{n}(a)$\\
$=2j(\Omega_{a,b}^1+\Omega_{b,a}^1)+2j(\Omega_{a,b}^3+\Omega_{b,a}^3)+2j(2\gamma^{-1}(\bar{\rho})(ab))+ $\\
$~~~~$ $\sum_{\alpha \in \overset{\circ} {\Delta}}\sum_{0< n\leq -j}(n+j)[x_{-\alpha}\otimes t^{n}(a), x_{\alpha}\otimes t^{-n}(b)]$\\
But $[x_{-\alpha}\otimes t^{n}(a), x_{\alpha}\otimes t^{-n}(b)]= -\gamma^{-1}(\alpha)+nK(ab)$\\
So the above sum $=2j(\Omega_{a,b}^1+\Omega_{b,a}^1+\Omega_{a,b}^3+\Omega_{b,a}^3)+2j(2\gamma^{-1}(\bar{\rho})(ab))+$\\
$~~~~$ $(dim~\mathfrak{g}-l)(\frac{j^3-j}{6})K(ab)$, where $ l = \mathrm{dim}\mathfrak{h}$.\\
For $k>0, k+j=0, [L_k, \sum_{n \in \mathbb{Z}}h_i \otimes t^{-n}(a)h^i \otimes t^{n+j}(b)]=2j(\Omega_{a,b}^2+\Omega_{b,a}^2)+l\frac{j^3-j}{6}K(ab),$\\
$[L_k, K(a)L_j (b) + K(b)L_j (a) + 2 h^{\vee}L_{j}(ab)] = 2j K(a) d(b) + 2j K(b)d(a) + 4j h^{\vee} d(ab) + \delta_{j +k, 0} \frac{k^3-k}{12}(2 K(a)K(b) + 2 h^{\vee} K(ab))$, where $d = L_0$.
So $[L_k, T_{-k}(a,b)]= 2j \Omega(a,b) + \frac{(j^3 - j)}{6}K(ab)\,\, \mathrm{dim}\,\,\mathfrak{g} + \delta_{j +k, 0} \frac{k^3-k}{12}(2 K(a)K(b) + 2 h^{\vee} K(ab))$ (see above 3.3).
\end{proof}
\begin{rmk}\rm
In this remark we explain an application of Proposition 3.1. Suppose $V(\psi)$ is an irreducible integrable highest weight module for $\tau(A)$. Then $V(\psi)$ is completely reducible module for $\bar{\mathfrak{g}}$. See Kac [11]. In fact $V(\psi)$ is sum of highest weight modules for $\bar{\mathfrak{g}}$.
One of the interesting questions is to to find multiplicities of these highest weights. Suppose $w \in V(\psi)$ is a highest weight vector, then $T_{j}(a,b)$ is also highest weight vector for $\bar{\mathfrak{g}}$  which follows from Theorem 3.2( for all $j$ and for all $a,b \in A$). In particular $T_{j}(a,b) v$ where $v$ is the genrator of $V(\psi)$ as $\tau(A)$ module for all $a,b \in A$ and for all $j$
is an highest weight vector for $\bar{\mathfrak{g}}$. They will not exhaust all the highest weight vectors of $\bar{\mathfrak{g}}$ as weights of all these vectors look like $\psi - n\delta$ .
\end{rmk}

\begin{example} \rm
We end this paper by giving an example where we apply  our operators.  Take $\mathfrak{g}=\mathfrak{sl}_2$ and $A=\mathbb{C}[t,t^{-1}]$. Let ${X,Y,h}$ be a $\mathfrak{sl}_2$ copy where $[X,Y]=h,[h,X]=2X,[h,Y]=-2Y$.  We take non-degenerate bilinear form on  $\mathfrak{g}$ in the following way $(X,Y)=1$ and $(h,h)=2$ and the rest is zero.   ${\frac{h}{2}},{h}$ is a dual basis for $\mathbb{C}h$.  Let $\psi_1,\psi_2 \in \bar{\mathfrak{h}}^{*}$ and let $V(\psi_1)$ and $V(\psi_2)$ be irreducible highest weight modules for $\tau = Vir\rtimes\tilde{\mathfrak{g}}$.  Let $z_1,z_2$ be non zero distinct complex numbers.  Let $M_i=(t-z_i)$ be the maximal ideal generated by $t-z_i$ inside $A$.  Then $V(\psi_1)\otimes V(\psi_2)$ be the evaluation module for $\tau(A)$ defined by $X \otimes t^m.w _1\otimes w_2=z_{1}^m( X.w_1\otimes w_2)+z_{2}^m (w_1 \otimes X.w_2)$ where $w_i \in V(\psi_i),X \in \tau, m \in \mathbb{Z}$.  From (1.7) we have a surjective map $\tau \otimes A \rightarrow \tau \bigoplus \tau$ where kernal is $\tau \otimes M_1 \cap M_2$.  Clearly $V(\psi_1) \otimes V(\psi_2)$ is an irreducible module for $\tau \otimes A$ from above map. 
Let $P_1(t)=\frac{t-z_2}{z_1-z_2},P_2(t)=\frac{t-z_1}{z_2-z_1}$ these polynomials are very special in the sense $P_i(z_j)=\delta_{ij}$.  $X \otimes P_1(t)P_2(t) $ is zero on $V(\psi_1) \otimes V(\psi_2)$  as $P_1(t)P_2(t) \in M_1 \cap M_2$.
For example $X \otimes P_1(t)$ acts only on the first component of $V(\psi_1) \otimes V(\psi_2)$ see [5] for more general setup.  Let $\psi_i(h)=\lambda_i$ and $\psi_{i}(K)=c_i$ then it is straightforward calculation to see the following.   Here $v_i$ is the highest weight vector of $V(\psi_i)$.\\
$T_{-1}(P_1(t), P_2(t))v_1 \otimes v_2 = Yv_1 \otimes (X\otimes t^{-1})v_2+(X \otimes t^{-1})v_2 \otimes Y v_2+ \frac{\lambda_1}{2} v_1 \otimes (h \otimes t^{-1})v_2+ 
\frac{\lambda_2}{2}(h \otimes t^{-1})v_1 \otimes v_2+c_1v_1 \otimes d_{-1}v_2+c_{2}d_{-1}v_1 \otimes v_2$;\\
 $T_{-2}(P_1(t), P_2(t))v_1 \otimes v_2=Yv_1 \otimes (X \otimes t^{-2})v_2+(X \otimes t^{-2})v_1 \otimes Y v_2+ (Y \otimes t^{-1})v_1 \otimes (X \otimes t^{-1}) v_2 + 
 (X \otimes t^{-1})v_1 \otimes (Y \otimes t^{-1}) v_2+\frac{\lambda_1}{2} v_1 \otimes (h \otimes t^{-2})v_2+ \frac{\lambda_2}{2}(h \otimes t^{-2})v_1 \otimes v_1+ \frac{1}{2}(h \otimes t^{-1})v_1 \otimes (h \otimes t^{-1})v_2+ c_{1}v_{1} \otimes d_{-2}v_{2}+c_{2}d_{-2}v_1 \otimes v_2$.\\
One can also directly verify that the above vectors are highest weight vector for $\tilde{\mathfrak{g}}$ .  It is suffice to check for $X$ and $Y(1)$ as they generate $N^{+}$.
\end{example}

{\bf{Acknowledgments :}} I would like to thank Sachin S. Sharma and Ravi Teja for helping with the manuscript.

\end{document}